\date{}
\documentclass[a4paper]{article}
\usepackage{amsfonts}
\usepackage{amsmath}
\usepackage{amssymb}
\usepackage{graphicx}
\usepackage{epstopdf}
\usepackage{bm}
\usepackage{cite}
\usepackage{amsthm}
\providecommand{\U}[1]{\protect \rule{.1in}{.1in}}
\newtheorem{theorem}{Theorem}

\newtheorem{definition}[theorem]{Definition}
\newtheorem{example}{Example}

\newtheorem{lemma}[theorem]{Lemma}
\newtheorem{notation}[theorem]{Notation}

\newtheorem{proposition}[theorem]{Proposition}
\newtheorem{remark}[theorem]{Remark}

\newtheorem{fact}[theorem]{Fact}

\begin{document}

\title{Fractional Generalized KYP Lemma for \\Fractional Order System within Finite Frequency Range}
\author{Xiaogang~Zhu,
        and~Junguo~Lu
\thanks{Junguo Lu is with the School of Electronic Information
and Electrical Engineering, Shanghai Jiao Tong University, Shanghai,
200240 China}
\thanks{.}}

%

\maketitle
%

\begin{abstract}%
The celebrated GKYP is widely used in integer-order control system. However, when it comes to the fractional order system, there exists no such tool to solve problems. This paper prove the FGKYP which can be used in the analysis of problems in fractional order system. The $H_\infty$ and $L_\infty$ of fractional order system are
analysed based on the FGKYP.

\end{abstract}%
%

\tableofcontents

\section{INTRODUCTION}
One fundamental research approach in control system is called the frequency-domain method. To the view of frequency-domain, the objective of designing a control system
is to find an appropriate controller which makes the system satisfying some
frequency response qualities. The celebrated Kalman-Yakubovich-Popov (KYP) lemma
bridges between the frequency-domain methods and time-domain methods. The KYP lemma originates from Popov's criterion \cite{Popov1979Absolute}, giving a frequency condition for stability of a feedback system, and then was proved by Kalman \cite{Kalman1963Lyapunov} and Yakubovich \cite{Yakubovich1962Solution} that Popov's
frequency condition was equivalent to existence of a Lyapunov function of certain simple form. It has been regarded as one of the most basic tools in control systems because it only needs to check
one matrix in linear matrix inequality (LMI) instead of checking the entire frequency range in frequency domain inequality (FDI). The KYP lemma \cite{Rantzer1996Kalman}
states that, given matrices $A, B$, and a Hermitian matrix $M$, for $\forall \omega
\in\mathbb{R}\cup\{\infty\}$, the following inequlity
\begin{align*}
\begin{bmatrix}
(\text{j}\omega I-A)^{-1}B\\
I
\end{bmatrix}^*
M
\begin{bmatrix}
(\text{j}\omega I-A)^{-1}B\\
I
\end{bmatrix}
<0
\end{align*}
holds if and only if there exists a Hermitian matrix $P$ such that
\begin{align*}
\begin{bmatrix}
A & B \\
I & 0
\end{bmatrix}^*
\begin{bmatrix}
0 & P\\
P & 0
\end{bmatrix}\begin{bmatrix}
A & B \\
I & 0
\end{bmatrix}+M<0
\end{align*}

However, the KYP lemma has its limitation when it's applied for practical control problems. Generally, practical control problems require systems can satisfy different
performance index for different frequency range. So, the KYP lemma is not compatible with the practical requirement. Iwasaki \cite{Iwasaki2005Generalized} points out that most practical control problems, including digital filter design, sensitivity-shaping, open-loop shaping and structure/control design integration,
only need to be analysed in certain frequency range. It's because many practical signals concentrate the energy in one or some finite frequency range. For example,
most energy of seismic wave is concentrated in frequency range 0.3-8Hz \cite{Chen2010Finite}.

In order to analyse control problems within finite frequency range, classic methods can be divided roughly into three ways\cite{Li2016Overview}: classical control theory, frequency-weighted method
\cite{Zhou1995Frequency,Wang1999new} and analysing control problems within finite frequency range directly. The classical control theory, including PID (Proportion integration differentiation) and root locus,
is mainly focus on the zero-pole point. But it mainly solves problems of linear SISO (Single Input Single Output) systems and is mostly dependent on experience.
The essential approach of frequency-weighted method is this method transforms
the original system, which has the control problem within finite frequency range, into a complex system, which has the control problem within infinite frequency range.
However, this method doesn't solve control problems within finite frequency range directly and depends mostly on experience. The third method, analysing control problems within finite frequency range directly, is now the major method to analyse such problems. It mainly includes Gramian \cite{Zhou1996Robust} and generalized KYP.
Iwasaki developed the
KYP lemma into generalized KYP (GKYP) lemma in 2005 \cite{Iwasaki2005Generalized}.
 The GKYP lemma consider the finite frequency intervals which is flexible for various frequency ranges. With the development of convex optimization, LMI is successfully and widely applied in control system \cite{Zhou1996Robust,Nesterov1994Interior,Yu2002Robust}. GKYP plays a very important
 role in transforming control problems into convex optimization.

 Even though there exist numerous researches utilizing the GKYP lemma, most of them are confined within the integer-order system
\cite{Li2014heuristic,Agulhari2012LMI,Gao2011H,Ding2010Fuzzy,Li2015Frequency,Zhang2015Analysis,Hoang2008Lyapunov}.
Therefore, the main purpose of this paper is to generalize the GKYP. We will prove the fractional generalized KYP (FGKYP) which can be utilized in the fractional order system (FOS).

To the best of our knowledge, there exists no research on the proof of FGKYP, but some papers  base their research on the GKYP. Liang et al. are the first to use GKYP to solve $H_\infty$ of fractional order system\cite{Liang2015Bounded}. But they partly prove that GKYP can be used in the fractional order system and they give a sufficient
condition for $H_\infty$ of FOS with fractional order $(0,1)$. Then, Sabatier et al. improve the condition of linear matrix inequality (LMI), reducing the number of variables\cite{Farges2013H}. In the most recent time, $H_\infty$ output feedback
control problem of linear time-invariant FOS over finite
frequency range is studied by Wang et al., based on the GKYP\cite{Wang2016H}, but they
utilize the GKYP directly.

Similarly to the integer-order system\cite{Wang2016H,Yang2008Generalized,ElGhaoui1997cone,Hara2009Sum,Li2015Min,Li2014Reduced,Li2012Generalized,Long2014Fault,Long2013Fault},
it's significant to go a step further in the research of FGKYP in the reason that FGKYP can be used conveniently to solve many kinds of problems in fractional order systems. In this paper, we'll utilize the FGKYP to solve $H_\infty$ and $L_\infty$ of FOS. Sabatier et al. are the first to compute $H_\infty$ norm of FOS,
and they use different kinds of methods to compute \cite{Sabatier2005Fractional,Fadiga2013H,Moze2008bounded,Fadiga2011computation}. $H_\infty$ of FOS is also used in other different problems, such as design of state feedback controller\cite{Farges2013H}, model match\cite{Padula2013H}
and model reduction\cite{Shen2013Reduced,Shen2014H}.

This paper is organized as follows.

In section II, the FOS model and the problem are stated. In section III, $S$-procedure
is introduced to bridge between matrix inequality and frequency range. In section IV, FGKYP for $L_\infty$ of FOS is proved and $L_\infty$ of FOS with finite frequency is studied. In section V, FGKYP for $H_\infty$ of FOS is proved and $H_\infty$ of FOS with finite frequency is studied. In section VI, numerical examples are given. Finally, in section VII, a conclusion is given.

\begin{notation}
$\nu$ is the order of the fractional order system (FOS), $\bm{\varphi}=\dfrac{\pi}{2}(\nu-1)$.
For a matrix $A$, its transpose, complex conjugate transpose are denoted by $A^{T}$ and $A^{\ast}$, respectively.
For matrices $A$ and $B$, $A\otimes B$ means the Kronecker product.
The conjugate of $x$ is denoted by $\overline{x}$. $\mathbb{R}$ and $\mathbb{C}$ denote real number
and complex number, respectively. $\mathbb{R}^+=\{x:\ x\in\mathbb{R},\ x\geq0\}$.
For $s\in\mathbb{C}$, Re$(s)$ denotes the real part of $s$ and Im$(s)$ denotes the imaginary part of $s$.
The convex hull and the interior of a set $\mathcal{X}$ are denoted
by $co(\mathcal{X})$ and $int(\mathcal{X})$, respectively. $\mathcal{H}_{n}$
stands for the set of $n\times n$ Hermitian matrices. For a matrix
$X\in \mathcal{H}_{n}$, inequalities $X>(\geq)0$ and $X<(\leq)0$ denote positive
(semi)definiteness and negative (semi)definiteness, respectively. The set
$\mathcal{J}$ denotes matrices $J=J^{\ast}\leq0$. $Sym(A)$ stands for
$A+A^{\ast}$. The null space of $X$ is denoted by $X_{\perp}$,i.e.,
$XX_{\perp}=0_{n}$. For $A\in%
\mathbb{C}
^{n\times m}$ and $B\in \mathcal{H}_{n+m}$, a function $\rho:%
\mathbb{C}
^{n\times m}\times \mathcal{H}_{n+m}\rightarrow \mathcal{H}_{m}$ is defined by%
\begin{equation}
\rho(A,B)\triangleq \left[
\begin{array}
[c]{c}%
A\\
I_{m}%
\end{array}
\right]  ^{\ast}B\left[
\begin{array}
[c]{c}%
A\\
I_{m}%
\end{array}
\right]  \label{DefChi}%
\end{equation}

\end{notation}

\section{Preliminaries}

\subsection{Fractional Order System(FOS) Model}

In this paper, the FOS is considered as follows
\begin{equation}
\left \{
\begin{array}
[c]{c}%
D^{\nu}x(t)=Ax(t)+Bu(t)\\
y(t)=Cx(t)+Du(t)
\end{array}
\  \  \right.  \label{EqFOS}%
\end{equation}
where $x(t)\in%
\mathbb{R}
^{n}$ is the pseudo state vector, $u(t)\in%
\mathbb{R}
^{n_{u}}$ is the control vector, $y(t)\in%
\mathbb{R}
^{n_{y}}$ is the sensed output, $\nu$ is the order of the fractional
order system and $0<\nu<2$. $A,B,C,D$ are constant real matrices. $D^{\nu}$ is the fractional
differentiation operator of order $\nu$. If the FOS is relaxed at $t=0$, transfer function matrix between $u(t)$ and
$y(t)$ is
\begin{equation}
G(s)=C(s^{\nu}I-A)^{-1}B+D \label{equTransFOS}
\end{equation}
\subsection{Problem Statement}

Motivated by finite frequency problems of digital filter design,
sensitivity-shaping, et.al, Iwasaki and Hara developed the KYP Lemma into the
GKYP Lemma \cite{Iwasaki2005Generalized}. The KYP lemma can check infinite frequency domain
inequality (FDI) via linear matrix inequality (LMI), whereas the GKYP Lemma can
check finite FDI via LMI. Given matrices $A,B,$ and a Hermitian matrix $\Pi$
and $\forall \omega \in%
\mathbb{R}
\cup \{ \infty \}$, the infinite FDI is described as%
\begin{equation}
\left[
\begin{array}
[c]{c}%
(\text{j}\omega I-A)^{-1}B\\
I
\end{array}
\right]  ^{\ast}\Pi \left[
\begin{array}
[c]{c}%
(\text{j}\omega I-A)^{-1}B\\
I
\end{array}
\right]  <0
\end{equation}

When it comes to FOS, it also exits problems which should be solved in
infinite frequency domain. Given matrices $A,B$ and a Hermitian matrix $\Pi$
and $\forall \omega \in%
\mathbb{R}
\cup \{ \infty \},$ the infinite FDI of FOS is described as%
\begin{equation}
\left[
\begin{array}
[c]{c}%
((\text{j}\omega)^{\nu}I-A)^{-1}B\\
I
\end{array}
\right]  ^{\ast}\Pi \left[
\begin{array}
[c]{c}%
((\text{j}\omega)^{\nu}I-A)^{-1}B\\
I
\end{array}
\right]  <0 \label{EqLDI}%
\end{equation}

Where $\nu$ is the fractional system order of the system.

As for FOS, we also want to check the finite FDI via LMI.

\section{$S$-procedure and Frequency Range}

$S$-procedure is stated as the following. Given $\Pi,F\in \mathcal{H}_{q}$, we
get the equivalence
\[
\eta^{\ast}\Pi \eta \leq0\  \  \forall \eta \in%
\mathbb{C}
^{q}\ such\ that\  \eta^{\ast}F\eta \geq0
\]
\[
\Leftrightarrow \exists \delta \in%
\mathbb{R}
\ such\ that\  \delta \geq0,\Pi+\delta F\leq0
\]

Where the regularity, $F\nleq0,$ is assumed. The strict inequality version
\[
\eta^{\ast}\Pi \eta<0\  \forall \eta \in%
\mathbb{C}
^{q}\ such\ that\  \eta^{\ast}F\eta \geq0
\]
\[
\Leftrightarrow \exists \delta \in%
\mathbb{R}
\ such\ that\  \delta \geq0,\Pi+\delta F<0
\]

The purpose of the $S$-procedure is to replace the former condition by the
latter condition because the latter condition is easier to verify. As for
FDI, the $S$-procedure bridge between the matrix inequality and frequency range.

In order to generalize the above $S$-procedure, paper\cite{Iwasaki2005Generalized} rewrites them with
different notation%
\begin{equation}
tr(\Pi \mathcal{G}_{1})\leq0\Leftrightarrow(\Pi+\mathcal{F})\cap \mathcal{J\neq
\emptyset} \label{EqSproc}%
\end{equation}
\begin{equation}
tr(\Pi \mathcal{G}_{1})<0\Leftrightarrow(\Pi+\mathcal{F})\cap int(\mathcal{J)\neq
\emptyset} \label{EqStrictSP}%
\end{equation}
where
\begin{equation}%
\begin{array}
[c]{c}%
\mathcal{F}\triangleq \left \{  \delta F:\delta \in%
\mathbb{R}
,\ \delta \geq0,\ F\in \mathcal{H}_{q}\right \} \\
\mathcal{G(F)}\triangleq \left \{  G\in \mathcal{H}_{q}:G\neq0,\ G\geq0,\ tr(\mathcal{F}G)\geq0\right \} \\
\mathcal{G}_{1}\mathcal{(F)}\triangleq \left \{  G\in \mathcal{G(F)}%
:rank(G)=1\right \}
\end{array}
\label{DefFG}%
\end{equation}

Paper\cite{Iwasaki2005Generalized} has already proved the lossless conditon for $S$-procedure as following.

First, the meaning of \textit{admissible, regular} and \textit{rank-one separable} is
given as follows.
\begin{definition}
A set $\mathcal{F}\subset \mathcal{H}_{q}$ is said to be
\begin{enumerate}
  \item \textit{admissible} if it is a nonempty closed convex cone and
$int(\mathcal{J})\cap \mathcal{F}=\emptyset$;
  \item \textit{regular }if $\mathcal{J}\cap \mathcal{F}=\left \{  0\right \}  $;
  \item \textit{rank-one separable} if $\mathcal{G}=co(\mathcal{G}_{1})$.
\end{enumerate}
\end{definition}

\begin{lemma}
[$S$-procedure]\label{LemSProc}
\cite{Iwasaki2005Generalized}
Let an admissible set $\mathcal{F}%
\subset \mathcal{H}_{q}$ be given and define $\mathcal{G}_{1}$ by (\ref{DefFG}).
Then, the strict $S$-procedure is lossless, i.e. (\ref{EqStrictSP}) holds for an arbitrary $\Pi\in\mathcal{H}_{q}$, if
and only if $\mathcal{F}$ is rank-one separable. Moreover, assuming that $\mathcal{F}$ is regular, then the nonstrict $S$-procedure
is lossless, i.e. (\ref{EqSproc}) holds for an arbitrary $\Pi\in\mathcal{H}_{q}$, if and only if $\mathcal{F}$ is
rank-one separable.

\end{lemma}

\begin{remark}
This lemma shows that when we choose an appropriate $\mathcal{F}$, which is
rank-one separable, the $S$-procedure will be lossless regardless of the
choice of $\Pi.$
\end{remark}

Paper\cite{Iwasaki2005Generalized} also gives some examples of admissible, regular and rank-one separable sets, which are readily proved.%
\begin{equation}
\mathcal{F}_{X}\triangleq \left \{  \left[
\begin{array}
[c]{cc}%
0 & X\\
X & 0
\end{array}
\right]  :X\in \mathcal{H}_{n}\right \}  \label{DefFx}%
\end{equation}

\begin{equation}
\mathcal{F}_{XY}\triangleq \left \{  \left[
\begin{array}
[c]{cc}%
-Y & X\\
X & Y
\end{array}
\right]  :X,Y\in \mathcal{H}_{n},Y\geq0\right \}  \label{DefFxy}%
\end{equation}

\begin{lemma}\label{lemRankone}
\cite{Iwasaki2005Generalized}
Let $\mathcal{F}\subset\mathcal{H}_m$ be a rank-one separable set. Then the set
$N^*\mathcal{F}N+\mathcal{P}$ is rank-one separable for any matrix $N\in\mathbb{C}^{m\times n}$
and subset $\mathcal{P}\subset\mathcal{H}_n$ of positive-semidefinite matrices containing the origin.
\end{lemma}

In general, a frequency range can be  visualized as a curve (or curves) on the
complex plane. Paper\cite{Iwasaki2005Generalized} define a curve as the following.

\begin{definition}
A curve on the complex plane is a collection of infinitely many points
$\theta(t)\in%
\mathbb{C}
$ continuously parametrized by $t$ for $t_{0}\leq t\leq t_{f}$ where $t_{0}%
$,$t_{f}\in%
\mathbb{R}
\cup \left \{  \pm \infty \right \}  $ and $t_{0}<t_{f}$. A set of complex numbers
$\mathbf{\Theta}\subseteq%
\mathbb{C}
$\ is said to represent a curve (or curves) if it is a union of a finite
number of curve(s). With $\Delta,\Sigma \in \mathcal{H}_{2}$ being given
matrices, $\mathbf{\Theta}$ is defined as:%
\begin{equation}
\mathbf{\Theta}(\Delta,\Sigma)\triangleq \left \{  \theta\in%
\mathbb{C}
|\ \rho(\theta,\Delta)=0,\ \rho(\theta,\Sigma)\geq0\right \}
\label{DefCurve}%
\end{equation}

\end{definition}

\begin{remark}
Note that the set $\mathbf{\Theta}(\Delta,\Sigma)$ is the intersection of
$\mathbf{\Theta}(\Delta,0)$ and $\mathbf{\Theta}(0,\Sigma)$. It can readily be
verified that the set $\mathbf{\Theta}(\Delta,0)$ represents a curve if and
only if $\det(\Delta)<0.$
\end{remark}

\begin{lemma}
\cite{Iwasaki2005Generalized}
Consider the set $\mathbf{\Theta}(\Delta,\Sigma)$ in (\ref{DefCurve}) and suppose it
represents curves on the complex plane. Then the
set $\mathbf{\Theta}(\Delta,\Sigma)$ is unbounded if and only if $\Delta_{11}=0$ and $\Sigma_{11}\geq0$.
\end{lemma}

\begin{lemma}
\label{LemTrans} Let $\Delta,\Sigma \in \mathcal{H}_{2}$ be given. Suppose
$\det(\Delta)<0$, then, there exits a common congruence
transformation such that%
\begin{equation}
\Delta=T^{\ast}\Delta_{0}T\  \  \  \Sigma=T^{\ast}\Sigma_{0}T
\label{equDelSig}
\end{equation}
\begin{equation}
\Delta_{0}\triangleq \left[
\begin{array}
[c]{cc}%
0 & e^{\text{j}\bm{\varphi}}\\
e^{-\text{j}\bm{\varphi}} & 0
\end{array}
\right]  \  \  \Sigma_{0}\triangleq \left[
\begin{array}
[c]{cc}%
\alpha & \beta e^{\text{j}\bm{\varphi}}\\
\beta e^{-\text{j}\bm{\varphi}} & \gamma
\end{array}
\right]  \label{DefDelt0}%
\end{equation}
where $\alpha,\beta,\gamma \in%
\mathbb{R}
$ and $T\in%
\mathbb{C}
^{2\times2}$. In particular, $\alpha$ and $\gamma$ can be ordered to satisfy
$\alpha \leq \gamma$.
\end{lemma}

Before we prove this Lemma, we prove the following Lemma first.

\begin{lemma}
Let $Y\in \mathcal{H}_{2}$ be given. Then, $Y$ admits the following
factorization:%
\begin{equation}
Y=L^{\ast}\left[
\begin{array}
[c]{cc}%
\alpha & \beta e^{\text{j}\bm{\varphi}}\\
\beta e^{-\text{j}\bm{\varphi}} & \gamma
\end{array}
\right]  L
\end{equation}

where $\alpha,\beta,\gamma \in%
\mathbb{R}
$, $L\in \mathcal{L}$ with
\[
\mathcal{L}\triangleq \{
Q^{\ast}ZQ:Z\in%
\mathbb{R}
^{2\times2},\det(Z)=1
\} \  \ Q\triangleq \left[
\begin{array}
[c]{cc}%
1 & 0\\
0 & je^{\text{j}\bm{\varphi}}%
\end{array}
\right]  .
\]
In particular, $\alpha$ and $\gamma$ are the eigenvalues of real matrix $Y_{0}%
\triangleq \left[
\begin{array}
[c]{cc}%
x & y\\
y & z
\end{array}
\right]  .$
\end{lemma}

\begin{proof}
Choose $Y$ as%
\begin{equation}
Y=\left[
\begin{array}
[c]{cc}%
x & (\beta+jy)e^{\text{j}\bm{\varphi}}\\
(\beta-jy)e^{-\text{j}\bm{\varphi}} & z
\end{array}
\right]
\end{equation}

where\ $\beta,x,y,z\in%
\mathbb{R}
.$

Then%
\begin{equation}
Y_{0}=Q\left(  Y-\left[
\begin{array}
[c]{cc}%
0 & \beta e^{\text{j}\bm{\varphi}}\\
\beta e^{-\text{j}\bm{\varphi}} & 0
\end{array}
\right]  \right)  Q^{\ast} \label{EqY0}%
\end{equation}

Since $Y_{0}$ is real symmetric, the spectral factorization of $Y_{0}$ gives%
\begin{equation}
Y_{0}=Z^{T}\left[
\begin{array}
[c]{cc}%
\alpha & 0\\
0 & \gamma
\end{array}
\right]  Z
\end{equation}

where the columns of $Z^{T}$ are eigenvectors and $\alpha$ and $\gamma$ are
eigenvalues. Moreover, $Z$ can be chosen to satisfy $\det(Z)=1.$ Then,
$L\triangleq Q^{\ast}ZQ$ belongs to $\mathcal{L.}$ Now, from (\ref{EqY0}) we
get%
\begin{align}
Y  &  =Q^{-1}Y_{0}Q^{\ast-1}+\beta \left[
\begin{array}
[c]{cc}%
0 & e^{\text{j}\bm{\varphi}}\\
e^{-\text{j}\bm{\varphi}} & 0
\end{array}
\right] \nonumber \\
&  =L^{\ast}\left[
\begin{array}
[c]{cc}%
\alpha & 0\\
0 & \gamma
\end{array}
\right]  L+\beta \left[
\begin{array}
[c]{cc}%
0 & e^{\text{j}\bm{\varphi}}\\
e^{-\text{j}\bm{\varphi}} & 0
\end{array}
\right]
\end{align}

Finally, it can readily be verified that%
\begin{equation}
L^{\ast}\left[
\begin{array}
[c]{cc}%
0 & e^{\text{j}\bm{\varphi}}\\
e^{-\text{j}\bm{\varphi}} & 0
\end{array}
\right]  L=\left[
\begin{array}
[c]{cc}%
0 & e^{\text{j}\bm{\varphi}}\\
e^{-\text{j}\bm{\varphi}} & 0
\end{array}
\right]
\end{equation}

holds for any $L\in \mathcal{L.}$

Therefore, we now can obtain the result%
\[
Y=L^{\ast}\left[
\begin{array}
[c]{cc}%
\alpha & \beta e^{\text{j}\bm{\varphi}}\\
\beta e^{-\text{j}\bm{\varphi}} & \gamma
\end{array}
\right]  L
\]

\end{proof}

\begin{proof}
[proof of Lemma \ref{LemTrans}]Since $\det(\Delta)<0$, there exists a nonsingular matrix $K$
such that%
\begin{equation}
\Delta=K^{\ast}\left[
\begin{array}
[c]{cc}%
0 & e^{\text{j}\bm{\varphi}}\\
e^{-\text{j}\bm{\varphi}} & 0
\end{array}
\right]  K
\end{equation}

holds. Let $Y\triangleq K^{\ast-1}\Sigma K^{-1}$, then we get%
\begin{equation}
\Sigma=K^{\ast}L^{\ast}\left[
\begin{array}
[c]{cc}%
\alpha & \beta e^{\text{j}\bm{\varphi}}\\
\beta e^{-\text{j}\bm{\varphi}} & \gamma
\end{array}
\right]  LK
\end{equation}

Therefore, the Lemma \ref{LemTrans} is proved by defining $T\triangleq LK$.
Since $\alpha,\gamma$ are the eigenvalues of $Y_{0}$, they can be ordered so
that $\alpha \leq \gamma.$
\end{proof}

\begin{lemma}\cite{Iwasaki2005Generalized}
Let $\Delta_{0},\Sigma_{0}\in \mathcal{H}_{2}$ and nonsingular $T\in%
\mathbb{C}
^{2\times2}$ be given. Define scalars $a,b,c$ and $d$ and function $E(s)$ by%
\begin{equation}
\left[
\begin{array}
[c]{cc}%
a & b\\
c & d
\end{array}
\right]  \triangleq T\  \    E(s)\triangleq \  \dfrac{b-ds}{cs-a}%
\end{equation}

Then, the following holds true%
\begin{align}
&  \left \{  \theta\in%
\mathbb{C}
:\theta\in \mathbf{\Theta}(T^{\ast}\Delta_{0}T,\ T^{\ast}\Sigma
_{0}T),\ c\theta+d\neq0\right \} \nonumber \\
&  =\{E(s)\in%
\mathbb{C}
:s\in \mathbf{\Theta}(\Delta_{0},\Sigma_{0}),\ cs\mathcal{\neq}a\}
\end{align}

\end{lemma}

\begin{remark}
This Lemma shows that
$\mathbf{\Theta}(T^{\ast}\Delta_{0}T,T^{\ast}\Sigma_{0}T)$
represents curve(s) if and only if $\mathbf{\Theta}(\Delta_{0},\Sigma_{0})$
dose so. When $\mathbf{\Theta}(\Delta,\Sigma)=\mathbf{\Theta}(T^{\ast}\Delta_{0}T,T^{\ast}\Sigma_{0}T)$,
$\mathbf{\Theta}(\Delta,\Sigma)$ represents
curve(s) if and only if $\mathbf{\Theta}(\Delta_{0},\Sigma_{0})$ dose so.
\end{remark}

Now, we examine the set $\mathbf{\Theta}(\Delta_0,\Sigma_0)$ with $\Delta_0$ and $\Sigma_0$ defined in (\ref{DefDelt0}).
Note that $\rho(\theta,\Delta_{0})=0$ holds if and only if
$\theta=\text{j}^{\nu}W$ for some $W \in%
\mathbb{R}
$. For such $\theta$, $\rho(\theta,\Sigma_{0})=\alpha W^2+\gamma$. $\mathbf{\Theta}(\Delta_{0},\Sigma_{0})$ represents curve(s) on the
complex plane, thus $\rho(\theta,\Sigma_{0})\geq0$ holds for some $W \in%
\mathbb{R}
$, which implies either $0\leq\alpha\leq\gamma$ or $\alpha<0<\gamma$.

\begin{proposition}\label{propCurve}
Let $\Delta,\Sigma \in \mathcal{H}_{2}$ be given and define the set
$\mathbf{\Theta}(\Delta,\Sigma )$ by (\ref{DefCurve}). For
FOS, the set $\mathbf{\Theta}(\Delta,\Sigma)$ represents
curve(s) on the complex plane if and only if the following two conditions hold:
\begin{itemize}
  \item $\det(\Delta)<0$
  \item either $0\leq\alpha\leq\gamma$ or $\alpha<0<\gamma$
\end{itemize}
where $\alpha,\beta$ and $\gamma$ are defined in (\ref{DefDelt0}).
\end{proposition}

Frequency $\omega\geq0$ and let $\omega$ belong to the principal Riemann surface \cite{Beyer1995Definition}, thus $\omega^{\nu}\in\mathbb{R}^{+}$.
Let $W=\omega^{\nu}$, then $W(\omega)$ is a monotone increasing function.
It's obvious that the set $\mathbf{\Theta}$ can represent a certain
range of the frequency variable $\theta$. For the continuous-time setting, we get%

\begin{equation}
\Delta=\left[
\begin{array}
[c]{cc}%
0 & e^{\text{j}\bm{\varphi}}\\
e^{-\text{j}\bm{\varphi}} & 0
\end{array}
\right]  \  \  \mathbf{\Theta}=\left \{  (\text{j}\omega)^{\nu}:\omega \in \mathbf{\Omega
}\right \}
\end{equation}

where $\mathbf{\Omega}$ is a subset of real numbers which is specified by an
additional choice of $\Sigma$, for example, as follows:
\[%
\begin{tabular}
[c]{|c|c|c|c|}
\hline
& LF & MF & HF\\
\hline
$\Omega$ & $0\leq\omega \leq \omega_{L}$ & $0\leq \omega
_{1}\leq \omega \leq \omega_{2}$ & $\omega \geq \omega_{H}\geq0$\\
\hline
$\Sigma$ & $\left[
\begin{array}
[c]{cc}%
-1 & 0\\
0 & \omega_{L}^{2\nu}%
\end{array}
\right]  $ & $\left[
\begin{array}
[c]{cc}%
-1 & \omega_{c}\\
\overline{\omega}_{c} & -\omega_{1}^{\nu}\omega_{2}^{\nu}%
\end{array}
\right]  $ & $\left[
\begin{array}
[c]{cc}%
1 & 0\\
0 & -\omega_{H}^{2\nu}%
\end{array}
\right]  $\\
\hline
\end{tabular}
\]
where $\omega_{c}=\dfrac{\text{j}^{\nu}(\omega_{1}^{\nu}+\omega_{2}^{\nu})}{2}$, and LF, MF, HF stand for low, middle, high frequency ranges, respectively.

\begin{remark}
We now can see that the main technical steps to arrive at the FGKYP lemma for
finite frequency FOS are to choose an appropriate set $\mathcal{F}$.
\end{remark}

\section{FGKYP for $L_{\infty}$ Norm of FOS}

\subsection{Main Theorem}

For the FDI in (\ref{EqLDI}), the set $\mathcal{G}_{1}$ should be given as
\begin{equation}
\mathcal{G}_{1}=\left \{  \eta \eta^{\ast}:\eta=\left[
\begin{array}
[c]{c}%
((\text{j}\omega)^{\nu}I-A)^{-1}B\\
I
\end{array}
\right]  \zeta,%
\begin{array}
[c]{c}%
\zeta \in%
\mathbb{C}
^{m},\zeta \neq0\\
\omega \in%
\mathbb{R}^{+}
\cup \{ \infty \}
\end{array}
\right \} \label{defG1}
\end{equation}

This set can be described as%
\begin{align}
\mathcal{G}_{1} & =\left \{  \eta \eta^{\ast}:\eta \in \mathcal{M}_{\theta},\ \theta
\in \overline{\mathbf{\Theta}}\right \}\notag\\
\mathcal{M}_{\theta} & \triangleq \left \{  \eta \in%
\mathbb{C}
^{n+m}:\eta \neq0,\Xi_{\theta}N\eta=0\right \}  \label{DefGMsita}%
\end{align}
where $\overline{\mathbf{\Theta}}\triangleq (\text{j}\mathbb{R}^+)^{\nu}\cup \{ \infty \}$ and
\begin{equation}
\Xi_{\theta}\triangleq \left \{
\begin{array}
[c]{cc}%
\left[  I_{n}\ -\theta I_{n}\right]  & (\theta \in%
\mathbb{C}
)\\
\left[  0\ -I_{n}\right]  & (\theta=\infty)
\end{array}
\right.\text{,}\ \
N\triangleq \left[
\begin{array}
[c]{cc}%
A & B\\
I_{n} & 0
\end{array}
\right]
\label{DefXitheta}%
\end{equation}

Therefore, when $\mathbf{\Theta}$ is defined in (\ref{DefCurve}),
$\overline{\mathbf{\Theta}}$ is defined as
\begin{equation}
\overline{\mathbf{\Theta}}\triangleq \left \{
\begin{tabular}
[c]{cc}%
$\mathbf{\Theta}$, & if $\mathbf{\Theta}$ is bounded\\
$\mathbf{\Theta}\cup \left \{  \infty \right \}$, & otherwise
\end{tabular}
\  \  \  \  \  \  \  \  \right.  \label{DefOSita}%
\end{equation}

Now, the main steps to obtain FGKYP lemma for FOS are to choose an appropriate
set $\mathcal{F}$ in (\ref{DefFG}) and then express $\mathcal{G}_1$ in (\ref{DefGMsita}) as in (\ref{DefFG}),
which should led to the result that the $S$-procedure is lossless.

\begin{lemma}\label{lemEquivtPsi}
\cite{Iwasaki2005Generalized}
Let $\Delta_0,\ \Sigma_0\in \mathcal{H}_2$ and a nonsingular matrix $T\in\mathbb{C}^{2\times 2}$ be given
and define $\Delta,\Sigma\in\mathcal{H}_2$ by (\ref{equDelSig}). Consider $\Xi_{\theta}$ in (\ref{DefXitheta}),
$\mathbf{\Theta}(\Delta,\Sigma)$ in (\ref{DefCurve}) and $\overline{\mathbf{\Theta}}(\Delta,\Sigma)$ in (\ref{DefOSita}).
Suppose $\mathbf{\Theta}(\Delta,\Sigma)$ represents curve(s). The following conditions on a given vector $\psi\in\mathbb{C}^{2n}$
are equivalent.

i) $\Xi_{\theta}\psi=0$ holds for some $\theta\in\overline{\mathbf{\Theta}}(\Delta,\Sigma)$.

ii) $\Xi_{s}(T\otimes I)\psi=0$ holds for some $s\in\overline{\mathbf{\Theta}}(\Delta_0,\Sigma_0)$.
\end{lemma}

\begin{lemma}\label{lemFGeq0}
\cite{Rantzer1996Kalman}
Let $F,G$ be complex matrices of the same size. Then
\[
FG^{\ast}+GF^{\ast}=0
\]
if and only if there exists a matrix $U$ such that $UU^{\ast}=I$ and $F(I+U)=G(I-U)$.
\end{lemma}

From lemma \ref{lemFGeq0}, we get the following lemma.
\begin{lemma}\label{lem_fjwg}
Let $f,g\in\mathbb{C}^{n}$ and $g\neq0$. Then
\[
e^{-\text{j}\bm{\varphi}}fg^{\ast}+e^{\text{j}\bm{\varphi}}gf^{\ast}=0\Leftrightarrow f=(\text{j}\omega)^{\nu}g\ for\ some\ \omega\in\mathbb{R}^+
\]
\end{lemma}
\begin{proof}
Let $F=f,G=e^{\text{j}\bm{\varphi}}g$ and $(1-U)/(1+U)=\text{j}\omega^{\nu}$ in Lemma \ref{lemFGeq0}, then we get the desired result.
\end{proof}

\begin{lemma}\label{lemEquivtZeta}
Let $\Delta_{0},\Sigma_{0}$ be defined in (\ref{DefDelt0}), $\mathbf{\Theta}$ in (\ref{DefCurve}) representing curves,
$\overline{\mathbf{\Theta}}$ in (\ref{DefOSita}) and $\Xi_{\theta}$ in
(\ref{DefXitheta}), then the following two conditions are equivalent

i) $\Xi_{s}\zeta=0$ for some $s\in \overline{\mathbf{\Theta}}(\Delta_{0},\Sigma_{0})$;

ii) $\zeta^{\ast}(\Delta_{0}\otimes U+\Sigma_{0}\otimes V)\zeta \geq0$ for all
$U,V\in \mathcal{H}_{n},V\geq0$

\end{lemma}

\begin{proof}
Define $\zeta=\left[  f^{\ast}\ g^{\ast}\right]  ^{\ast}$. Through some algebraic manipulations, we get
\begin{align}
& \zeta^{\ast}(\Delta_{0}\otimes U+\Sigma_{0}\otimes V)\zeta \notag\\
=& \alpha f^{\ast}Vf+\beta e^{-\text{j}\bm{\varphi}}g^{\ast}Vf+\beta e^{\text{j}\bm{\varphi}}f^{\ast}Vg+
\gamma g^{\ast}Vg+e^{-\text{j}\bm{\varphi}}g^{\ast}Uf+e^{\text{j}\bm{\varphi}}f^{\ast}Ug \notag \\
=& \text{tr}\left[
(\alpha ff^{\ast}+\beta e^{-\text{j}\bm{\varphi}}fg^{\ast}+\beta e^{\text{j}\bm{\varphi}}gf^{\ast}+\gamma gg^{\ast})V
\right]+\text{tr}\left[
(e^{-\text{j}\bm{\varphi}}fg^{\ast}+e^{\text{j}\bm{\varphi}}gf^{\ast})U
\right] \label{equPrfZata1}
\end{align}

Suppose i) holds.

It can be verified
that i) holds if and only if either a) $\overline{\mathbf{\Theta}}(\Delta_{0},\Sigma_{0})$ is bounded
and $f=(\text{j}\omega)^{\nu}g$ holds for some $\omega\in\mathbb{R}^+$ such that $\rho(\theta,\Delta_0)\geq0$ or b)
$\overline{\mathbf{\Theta}}(\Delta_{0},\Sigma_{0})$ is unbounded ($\alpha\geq0$) and
$g=0$.

If $f=(\text{j}\omega)^{\nu}g$, then
\[
\zeta^{\ast}(\Delta_{0}\otimes U+\Sigma_{0}\otimes V)\zeta
=(\alpha \omega^{2\nu}+\gamma)g^{\ast}Vg
\]

Because $\rho(\theta,\Delta_0)\geq0$ and $V\geq0$, we get
$\zeta^{\ast}(\Delta_{0}\otimes U+\Sigma_{0}\otimes V)\zeta \geq0$.

If $\alpha\geq0$ and $g=0$, it's obvious that $\zeta^{\ast}(\Delta_{0}\otimes U+\Sigma_{0}\otimes V)\zeta \geq0$
for all $U,V\in \mathcal{H}_{n},V\geq0$.

Suppose ii) is satisfied. It implies that
\begin{align}
\alpha ff^{\ast}+\beta e^{-\text{j}\bm{\varphi}}fg^{\ast}+\beta e^{\text{j}\bm{\varphi}}gf^{\ast}+\gamma gg^{\ast} \geq0\\
e^{-\text{j}\bm{\varphi}}fg^{\ast}+e^{\text{j}\bm{\varphi}}gf^{\ast}=0 \label{equPrfZata2}
\end{align}
both hold.

According to Lemma \ref{lem_fjwg}, equation (\ref{equPrfZata2}) implies that either $f=(\text{j}\omega)^{\nu}g$, $g\neq0$
 or $g=0$ holds. $f=(\text{j}\omega)^{\nu}g$ can further derive i) when $\mathbf{\overline{\Theta}(\Delta_0,\Sigma_0)}$ is bounded.
 $g=0$ can further derive i) when $\mathbf{\overline{\Theta}(\Delta_0,\Sigma_0)}$ is unbounded. This ends the proof.
\end{proof}

\begin{lemma}\label{lemRankoneF}
Let $N\in%
\mathbb{C}
^{2n\times(n+m)}$ and $\Delta,\Sigma \in \mathcal{H}_{2}$ be given such that
$\mathbf{\Theta}$ in (\ref{DefCurve}) represents curves. Define
$\overline{\mathbf{\Theta}}$ and $\Xi_{\theta}$ by (\ref{DefOSita}%
) and (\ref{DefXitheta}), respectively. Then, the set $\mathcal{G}_{1}$ defined in
(\ref{DefGMsita}) can be characterized by (\ref{DefFG}) with%
\begin{equation}
\mathcal{F}\triangleq \left\{
N^{\ast}(\Delta \otimes U+\Sigma \otimes V)N:U,V\in\mathcal{H}_n,V\geq0
\right\}
\label{DefFofGKYP}%
\end{equation}

\end{lemma}

\begin{proof}
Let $\mathcal{G}_{2}$ be defined to be $\mathcal{G}_{1}$ in (\ref{DefGMsita}) with
(\ref{DefFofGKYP}) and $N_{0}\triangleq(T\otimes I)N$. Then, for a nonzero
vector $\eta$
\begin{align*}
& \eta \eta^{\ast}\in \mathcal{G}_{1}\\
\Leftrightarrow & \Xi_{\theta}N\eta=0\text{ for
some }\theta\in \overline{\mathbf{\Theta}}(\Delta,\Sigma)\\
\Leftrightarrow & \Xi_{s}N_{0}\eta=0\text{ for some }s\in \overline
{\mathbf{\Theta}}(\Delta_{0},\Sigma_{0})\\
\Leftrightarrow & \eta^{\ast}N_{0}^{\ast}(\Delta_{0}\otimes U+\Sigma_{0}\otimes
V)N_{0}\eta \geq0\\
           & \text{for all }U,V\in\mathcal{H}_n,V\geq0\\
\Leftrightarrow & \eta \eta^{\ast}\in \mathcal{G}_{2}
\end{align*}
where the first and fourth equivalences can easily be gotten from the
definitions, and the second equivalence holds due to Lemma \ref{lemEquivtPsi}, and the
third equivalence holds due to the Lemma \ref{lemEquivtZeta}, respectively.
\end{proof}

Now we can get the rank-one separable set $\mathcal{F}$.
\begin{lemma}\label{lemRkAdm}
Let $N\in%
\mathbb{C}
^{2n\times(n+m)}$ and $\Delta,\Sigma \in \mathcal{H}_{2}$ be given such that
$\mathbf{\Theta}$ in (\ref{DefCurve}) represents curves.
Define $\overline{\mathbf{\Theta}}$ by
(\ref{DefOSita}), the set $\mathcal{F}$ by
(\ref{DefFofGKYP}) and the matrix $\Xi_{\theta}$ by (\ref{DefXitheta}). Then the set $\mathcal{F}$ is admissible and rank-one separable.
\end{lemma}

\begin{proof}
Clearly, $\mathcal{F}$ is a closed convex cone. When $F\in\mathcal{F}>0$, the set
$\mathcal{G}_{1}$ is nonempty and hence $\mathcal{F}$ is admissible (\cite{Iwasaki2005Generalized}, Lemma
11). From Lemma \ref{LemTrans}, we get
\begin{equation*}
\Delta \otimes U+\Sigma \otimes V=(T\otimes I)^{\ast}\left[
\begin{array}
[c]{cc}%
\alpha V & Ue^{\text{j}\bm{\varphi}}+\beta e^{\text{j}\bm{\varphi}}V\\
Ue^{-\text{j}\bm{\varphi}}+\beta e^{-\text{j}\bm{\varphi}}V & \gamma V
\end{array}
\right]  (T\otimes I)
\end{equation*}
where $\alpha \leq \gamma$ and $\gamma \geq0$ according to Proposition \ref{propCurve}.

When $\alpha<0<\gamma$, define
\[
W\triangleq \left[
\begin{array}
[c]{cc}%
\sqrt{-\alpha}Ie^{-\text{j}\bm{\varphi}/2} & 0\\
0 & \sqrt{\gamma}Ie^{\text{j}\bm{\varphi}/2}%
\end{array}
\right]  (T\otimes I)N
\]%
\[
X\triangleq \frac{(U+\beta V)}{\sqrt{-\alpha \gamma}},Y\triangleq V
\]

Then, the set $\mathcal{F}$ can be characterized as $\mathcal{F=}W^{\ast
}\mathcal{F}_{XY}W$ with $\mathcal{F}_{XY}$ defined in (\ref{DefFxy}).

When $\gamma \geq \alpha \geq0$, define
\[
K\triangleq \left[
\begin{array}
[c]{cc}%
e^{-\text{j}\bm{\varphi}/2} & 0\\
0 & e^{\text{j}\bm{\varphi}/2}%
\end{array}
\right]  (T\otimes I)N
\]%
\[
X\triangleq U+\beta V,P\triangleq((T\otimes I)N)^{\ast}\left[
\begin{array}
[c]{cc}%
\alpha V & 0\\
0 & \gamma V
\end{array}
\right]  (T\otimes I)N
\]

Then, we get $\mathcal{F=}K^{\ast}\mathcal{F}_{X}K+\mathcal{P}$ with
$\mathcal{F}_{X}$ defined in (\ref{DefFx}), and the set $\mathcal{P\triangleq
}\left \{  P\right \}  $ is obviously a subset of positive-semidefinite matrices
containing the origin.

Since $\mathcal{F}_{X}$ and $\mathcal{F}_{XY}$ are rank-one separable, it can be
verified that $\mathcal{F}$ is rank-one separable according to Lemma \ref{lemRankone}.
\end{proof}

Now, we are ready to state and prove the theorem for finite frequency FOS.

\begin{theorem}\label{thmKYPforFOS}
Let matrices $\Pi \in \mathcal{H}_{n+m}$, $N\in%
\mathbb{C}
^{2n\times(n+m)}$, and $\Delta,\Sigma \in \mathcal{H}_{2}$ be given and
$\mathbf{\Theta}$ and $\overline{\mathbf{\Theta}}$ is
defined by (\ref{DefCurve}) and (\ref{DefOSita}), respectively.
Suppose $\mathbf{\Theta}$ represents curves on the right half complex plane and
$\overline{\mathbf{\Theta}}$ represents $\mathbf{\Theta}\cup \{ \infty \}$.
$\Xi_{\theta}$ is defined in (\ref{DefXitheta}) and $S_{\theta}$ is defined as
$S_{\theta}\triangleq(\Xi_{\theta}N)_{\perp}$. The following statements are equivalent

i) $S_{\theta}^{\ast}\Pi S_{\theta}<0,$ $\forall \theta\in \overline
{\mathbf{\Theta}}(\Delta,\Sigma).$

ii) There exist $U,V\in \mathcal{H}_{n}$ such that $V>0$ and%
\begin{equation}
N^{\ast}(\Delta \otimes U+\Sigma \otimes V)N+\Pi<0
\label{ineqThmGKYP}
\end{equation}

\end{theorem}

\begin{proof}
Note that i) holds if and only if $\text{tr}(\Pi \mathcal{G}_1)<0$ holds where $\mathcal{G}_1$ is defined in (\ref{DefGMsita}).
The set $\mathcal{G}_1$ can be characterized by (\ref{DefFG}) with $\mathcal{F}$ in (\ref{DefFofGKYP}) according to Lemma \ref{lemRankoneF}.
By Lemma \ref{lemRkAdm}, the set $\mathcal{F}$ is admissible and rank-one separable, which means i) is equivalent to
$\mathcal{F}+\Pi<0$ according to Lemma \ref{LemSProc}, i.e. ii) holds. Because the inequality in (\ref{ineqThmGKYP}) is strict, the existence
of $V$ can be chosen as $V>0$ without loss of generality.
\end{proof}

\subsection{$L_{\infty}$ with Different Frequency Range}

The following gives the result of $L_{\infty}$ with finite frequency.

\begin{definition}
For a matrix function $T(s)$, the $L_{\infty}$ norm of $T(s)$ is defined as
\begin{equation}
\left \Vert T(s) \right \Vert_{L_{\infty}}\triangleq\underset{\omega\in\mathbb{R}}{\sup}\sigma_{\max}(T(\text{j}\omega))
\end{equation}
where $\sigma_{\max}$ is the maximum singular value.
\end{definition}

\begin{lemma}
\cite{Liang2015Bounded}
For a matrix function $T(s)$, there holds
\begin{equation}
\left \Vert T(s) \right \Vert_{L_{\infty}}=\underset{\omega\geq0}{\sup}\sigma_{\max}(T(\text{j}\omega))
\end{equation}
\end{lemma}

\begin{theorem}[LMI for FOS of Low Frequency]\label{LMIforLowF}
 Consider FOS with its transfer function $G(s)$ in (\ref{equTransFOS}).
Given a prescribed $L_{\infty}$ performance bound $\delta>0$, then
$\left \Vert G(s)\right \Vert _{L_{\infty}}=\underset{\omega}{\sup}%
\sigma_{\max}(G(\text{j}\omega)<\delta$, $\omega$ belong to the principal Riemann surface and
$\omega \in \Omega_{L}\triangleq \{
\omega \in%
\mathbb{R}^+
:\omega \leq \omega_{L}\}$, holds if and
only if there exist $U,V\in \mathcal{H}_{n},V>0$, such that
\begin{equation}
\left[
\begin{array}
[c]{ccc}
Sym(X)-A^TVA+\omega_{L}^{2\nu}V & Y^{\ast} & C^T\\
Y & -\delta I-B^TVB & D^T\\
C & D & -\delta I
\end{array}
\right]  <0 \label{EqLFFOS}
\end{equation}
where $X\triangleq e^{\text{j}\bm{\varphi}}A^TU$, $Y\triangleq-B^TVA+e^{\text{j}\bm{\varphi}}B^TU$,
 $\sigma_{\max}$ is the maximum singular value.
\end{theorem}

\begin{proof}
Let $\Delta=\left[
\begin{array}
[c]{cc}%
0 & e^{\text{j}\bm{\varphi}}\\
e^{-\text{j}\bm{\varphi}} & 0
\end{array}
\right]  $, $\Sigma=\left[
\begin{array}
[c]{cc}%
-1 & 0\\
0 & \omega_{L}^{2\nu}%
\end{array}
\right]  ,$ and then it can readily be verified that $\overline{\mathbf{\Theta
}}(\Delta,\Sigma)$ can represent a curve on the complex plane with
the frequency range $\Omega_{L}$. Let $\theta(\omega)\triangleq e^{\text{j}\frac
{\pi}{2}\nu}\omega^{\nu}$, then $G(\text{j}\omega)=C
(\theta(\omega)I-A)^{-1}B+D.$

By some basic matrix calculations, we get
\begin{align}
& \underset{\omega}{\sup}\sigma_{\max}(G(\text{j}\omega))<\delta \notag\\
\Leftrightarrow & G^{\ast}(\text{j}\omega)G(\text{j}\omega)-\delta^{2}%
I<0,\forall \omega \in \Omega_{L}\notag\\
\Leftrightarrow & \left[
\begin{array}
[c]{c}%
H(\theta)\\
I
\end{array}
\right]  ^{\ast}\Pi \left[
\begin{array}
[c]{c}%
H(\theta)\\
I
\end{array}
\right]  <0,\forall \theta\in \overline{\mathbf{\Theta}}
(\Delta,\Sigma) \label{EqHSita}
\end{align}
where $H(\theta)\triangleq(\theta I-A)^{-1}B$ and
\begin{equation}
\Pi \triangleq \left[
\begin{array}
[c]{cc}%
C^TC & C^TD\\
D^TC & D^TD-\delta^{2}I
\end{array}
\right]
\end{equation}

According to the theorem \ref{thmKYPforFOS}, the last part of (\ref{EqHSita}) is also
equivalent to the following LMI.
\begin{align*}
\begin{bmatrix}
A & B \\
I & 0
\end{bmatrix}^T\begin{bmatrix}
-V & e^{\text{j}\bm{\varphi}}U\\
e^{-\text{j}\bm{\varphi}}U & \omega_{L}^{2\nu}V
\end{bmatrix}\begin{bmatrix}
A & B \\
I & 0
\end{bmatrix}+\Pi<0
\end{align*}

This LMI can be simplified as%
\begin{align}
& \begin{bmatrix}
Sym(X)-A^TVA+\omega_{L}^{2\nu}V & Y^{\ast}\\
Y & -\delta^2 I-B^TVB
\end{bmatrix} + \begin{bmatrix}
C & D
\end{bmatrix}^T\begin{bmatrix}
C & D
\end{bmatrix}<0 \label{EqLFLMI}
\end{align}

where $X\triangleq e^{\text{j}\bm{\varphi}}A^TU$, $Y\triangleq-B^TVA+e^{\text{j}\bm{\varphi}}B^TU$.

Rescaling $U$,$V$ and utilizing the Schur complement theorem, (\ref{EqLFFOS}) is finally achieved.
\end{proof}

\begin{theorem}
[LMI for FOS of High Frequency] Consider FOS with its transfer function $G(s)$ in (\ref{equTransFOS}).
Given a prescribed $L_{\infty}$ performance bound $\delta>0$, then
$\left \Vert G(s)\right \Vert _{L_{\infty}}=\underset{\omega}{\sup}%
\sigma_{\max}(G(\text{j}\omega)<\delta$, $\omega$  belong to the principal Riemann surface and
$\omega \in \Omega_{H}\triangleq \{
\omega \in%
\mathbb{R}^+
:\omega \geq \omega_{H}\}$, holds if and
only if there exist $U,V\in \mathcal{H}_{n},V>0$, such that
\begin{equation}
\left[
\begin{array}
[c]{ccc}
Sym(X)+A^TVA-\omega_{H}^{2\nu}V & Y^{\ast} & C^T\\
Y & -\delta I+B^TVB & D^T\\
C & D & -\delta I
\end{array}
\right]  <0
\end{equation}
where $X\triangleq e^{\text{j}\bm{\varphi}}A^TU$, $Y\triangleq B^TVA+e^{\text{j}\bm{\varphi}}B^TU$,
 $\sigma_{\max}$ is the maximum singular value.
\end{theorem}

\begin{theorem}
[LMI for FOS of Middle Frequency]Consider FOS with its transfer function $G(s)$ in (\ref{equTransFOS}).
Given a prescribed $L_{\infty}$ performance bound $\delta>0$, then
$\left \Vert G(s)\right \Vert _{L_{\infty}}=\underset{\omega}{\sup}%
\sigma_{\max}(G(\text{j}\omega)<\delta$ , $\omega \in \Omega_{M}\triangleq \{
\omega \in%
\mathbb{R}^+
:\omega_1\leq\omega \leq \omega_{2}\}$, holds if and
only if there exist $U,V\in \mathcal{H}_{n},V>0$, such that
\begin{equation}
\left[
\begin{array}
[c]{ccc}
Sym(X)-A^TVA-\omega_{1}^{\nu}\omega_{2}^{\nu}V & Y^{\ast} & C^T\\
Y & -\delta I-B^TVB & D^T\\
C & D & -\delta I
\end{array}
\right]  <0
\end{equation}
where $X\triangleq A^T(e^{\text{j}\bm{\varphi}}U+\omega_cV)$, $Y\triangleq-B^TVA+B^T(e^{\text{j}\bm{\varphi}}U+\omega_cV)$,
$\omega_c=\text{j}^{\nu}\dfrac{\omega^{\nu}_1+\omega^{\nu}_2}{2}$, $\sigma_{\max}$ is the maximum singular value.
\end{theorem}

\begin{theorem}
[LMI for FOS of Infinite Frequency]Consider FOS with its transfer function $G(s)$ in (\ref{equTransFOS}).
Given a prescribed $L_{\infty}$ performance bound $\delta>0$, then
$\left \Vert G(s)\right \Vert _{L_{\infty}}=\underset{\omega}{\sup}%
\sigma_{\max}(G(\text{j}\omega)<\delta$ ,  $\omega$  belong to the principal Riemann surface and
$\omega \in \Omega_{I}\triangleq \mathbb{R}^+\cup \{+\infty\}$, holds if and
only if there exist $U,V\in \mathcal{H}_{n},V>0$, such that
\begin{equation}
\left[
\begin{array}
[c]{ccc}
Sym(X) & Y^{\ast} & C^T\\
Y & -\delta I & D^T\\
C & D & -\delta I
\end{array}
\right]  <0
\end{equation}
where $X\triangleq e^{\text{j}\bm{\varphi}}A^TU$, $Y\triangleq e^{\text{j}\bm{\varphi}}B^TU$,
 $\sigma_{\max}$ is the maximum singular value.
\end{theorem}


\begin{proof}
The theorem of high frequency and middle frequency can be proved similarly to the proof of
low frequency. The curve $\overline{\mathbf{\Theta}}(\Delta,\Sigma)$
in high frequency is chosen as%
\[
\Delta=\left[
\begin{array}
[c]{cc}%
0 & e^{\text{j}\bm{\varphi}}\\
e^{-\text{j}\bm{\varphi}} & 0
\end{array}
\right]  ,\Sigma=\left[
\begin{array}
[c]{cc}%
1 & 0\\
0 & -\omega_{H}^{2\nu}%
\end{array}
\right]
\]

The curve $\overline{\mathbf{\Theta}}(\Delta,\Sigma)$ in middle
frequency is chosen as%
\[
\Delta=\left[
\begin{array}
[c]{cc}%
0 & e^{\text{j}\bm{\varphi}}\\
e^{-\text{j}\bm{\varphi}} & 0
\end{array}
\right]  ,\Sigma=\left[
\begin{array}
[c]{cc}%
-1 & \text{j}^{\nu}\dfrac{\omega_{1}^{\nu}+\omega_{2}^{\nu}}{2}\\
(-\text{j})^{\nu}\dfrac{\omega_{1}^{\nu}+\omega_{2}^{\nu}}{2} &
-\omega_{1}^{\nu}\omega_{2}^{\nu}%
\end{array}
\right]
\]


The curve $\overline{\mathbf{\Theta}}(\Delta,\Sigma)$
for infinite frequency is chosen as%
\[
\Delta=\left[
\begin{array}
[c]{cc}%
0 & e^{\text{j}\bm{\varphi}}\\
e^{-\text{j}\bm{\varphi}} & 0
\end{array}
\right]  ,\Sigma=\mathbf{0}_2
\]

This ends the proof.
\end{proof}

\begin{remark}
For the infinite frequency range, when the fractional order $\nu=1$, the condition
is as the same as the KYP\cite{Iwasaki2005Generalized}.
Meanwhile, Liang\cite{Liang2015Bounded} proves a theorem of $L_\infty$ for infinite frequency, but he utilizes the GKYP lemma directly. The results are different because he chooses the $\Sigma$ as
\[
\Sigma=\begin{bmatrix}
0 & 1-\alpha \\
1-\alpha & 0
\end{bmatrix}
\]
but the two theorems are equivalent.
\end{remark}

\section{FGKYP for $H_{\infty}$ Norm of FOS}

In this section, we check the $H_{\infty}$ norm of FOS.
\begin{definition}
For a matrix function $T(s)$, the $H_{\infty}$ norm of $T(s)$ is defined as
\begin{equation}
\left \Vert T(s) \right \Vert_{H_{\infty}}\triangleq\underset{\text{Re}(s)\geq0}{\sup}\sigma_{\max}(T(s))
\end{equation}
where $\sigma_{\max}$ is the maximum singular value.
\end{definition}

When we want to check $\left \Vert G(s) \right \Vert_{H_{\infty}} < \delta$ where $G(s)$ is a transfer function in (\ref{equTransFOS}), we get the following LMI
\begin{equation}
\left[
\begin{array}
[c]{c}%
(s^{\nu}-A)^{-1}B\\
I
\end{array}
\right]  ^{\ast}\Pi \left[
\begin{array}
[c]{c}%
(s^{\nu}-A)^{-1}B\\
I
\end{array}
\right]  <0 \label{EqLDIs}%
\end{equation}
where $s\in\mathbb{C},\text{Re}(s)\geq0$. Now we use the $S$-Procedure to check this LMI condition.

\begin{definition}
A convex region described by two straight lines on the complex plane is defined as $\mathbf{\Theta}(\Delta,\Sigma)$.
$\Delta,\Sigma\in\mathcal{H}_2$ and $\Delta\triangleq\begin{bmatrix}
0 & \alpha \\
\overline{\alpha} & 0
\end{bmatrix}$, $\Sigma\triangleq\begin{bmatrix}
0 & \beta \\
\overline{\beta} & 0
\end{bmatrix}$.
\begin{equation}
\mathbf{\Theta}(\Delta,\Sigma)\triangleq\left\{
\theta\in\mathbb{C}:\begin{bmatrix}
\theta\\1
\end{bmatrix}^{\ast}\Delta\begin{bmatrix}
\theta\\1
\end{bmatrix}\geq0,\begin{bmatrix}
\theta\\1
\end{bmatrix}^{\ast}\Sigma\begin{bmatrix}
\theta\\1
\end{bmatrix}\geq0
\right\} \label{defConvexD}
\end{equation}
\end{definition}

Define set $\mathcal{G}_1$ in (\ref{defG1}) as
\begin{equation}
\mathcal{G}_{1}=\left \{  \eta \eta^{\ast}:\eta=\left[
\begin{array}
[c]{c}%
(s^{\nu}I-A)^{-1}B\\
I
\end{array}
\right]  \zeta,%
\begin{array}
[c]{c}%
\zeta \in%
\mathbb{C}
^{m},\zeta \neq0\\
s \in%
\mathbb{C}\cup\{\infty\},\text{Re}(s)\geq0
\end{array}
\right \}
\end{equation}
Then
\begin{align}
\mathcal{G}_{1} & =\left \{  \eta \eta^{\ast}:\eta \in \mathcal{M}_{\theta},\ \theta
\in \mathbf{\Theta}\right\} \notag\\
\mathcal{M}_{\theta} & \triangleq \left \{  \eta \in%
\mathbb{C}
^{n+m}:\eta \neq0,\Xi_{\theta}N\eta=0\right \}
\label{defG1s}
\end{align}
where
\begin{equation}
\Xi_{\theta}\triangleq \left \{
\begin{array}
[c]{cc}%
\left[  I_{n}\ -\theta I_{n}\right]  & (\theta \in%
\mathbb{C}
)\\
\left[  0\ -I_{n}\right]  & (\theta=\infty)
\end{array}
\right. \label{defXis}
\end{equation}
\begin{equation}
N\triangleq \left[
\begin{array}
[c]{cc}%
A & B\\
I & 0
\end{array}
\right]\label{defNs}
\end{equation}

Let $s$ belong to the principal Riemann surface, i.e. $\{ s\mid -\pi<\text{arg}(s)<\pi\}$, only on which the roots of
$\det (s^{\nu}I-A)=0$ decide the time-domain behavior and stability of fractional system \cite{Beyer1995Definition}.
Now, we need to find a rank-one separable set $\mathcal{F}$, which can satisfy (\ref{EqStrictSP}).

\subsection{For fractional order $0< \nu\leq 1$}

\begin{lemma}\label{lemFGgeq0}
\cite{Rantzer1996Kalman}
Let $f,g\in\mathbb{C}^n$ and $g\neq0$. Then
\begin{equation}
fg^*+gf^*\geq0\Leftrightarrow f=\theta g\ \text{for some }\theta\in\mathbb{C}\text{ with Re}(\theta)\geq0
\end{equation}
\end{lemma}

\begin{lemma}\label{lemZetaReS}
Let $\mathbf{\Theta}(\Delta,\Sigma)$ be defined by (\ref{defConvexD}), $\Xi_{\theta}$ by (\ref{defXis}) and
$\zeta$ is a given vector. If the region $\mathbf{\Theta}(\Delta,\Sigma)$
defined by (\ref{defConvexD}) represents the region $\Omega=\{s^{\nu}\mid s\in\mathbb{C},
\text{Re}(s)\geq0,0<\nu\leq1\}$ on the complex plane and
$s$ belongs to principal Riemann surface,
then the following statements are equivalent.

i) $\Xi_{\theta}\zeta=0$ for some $\theta\in\mathbf{\Theta}$ with $\text{Re}(\theta)\geq0$;

ii) $\zeta^*\left[(\Delta+\Sigma)\otimes U\right]\zeta\geq0$ for all $U\in\mathcal{H}_n,U>0$.
\end{lemma}

\begin{proof}
Let $\Delta=\begin{bmatrix}
0 & a+\text{j}c\\
a-\text{j}c & 0
\end{bmatrix}$ and $\Sigma=\begin{bmatrix}
0 & b+\text{j}d\\
b-\text{j}d & 0
\end{bmatrix}$, ($\alpha=a+\text{j}c$, $\beta=b+\text{j}d$), and $a,b,c,d\in\mathbb{R}$.
If $\mathbf{\Theta}$ represent the region $\Omega$ and
$\theta=x+\text{j}y\in\mathbf{\Theta},x,y\in\mathbb{R}$, then

\[
\mathbf{\Theta}=\{\theta=x+\text{j}y\mid\sin(\frac{\pi}{2}\nu)x+y\cos(\frac{\pi}{2}\nu)\geq0,\sin(\frac{\pi}{2}\nu)x-\cos(\frac{\pi}{2}\nu)y\geq0
\}
\]
i.e. $a=b=\sin(\frac{\pi}{2}\nu)>0,c=-d=\cos(\frac{\pi}{2}\nu)$.

Therefore, when $\mathbf{\Theta}$ represent the region $\Omega$, there holds $\alpha+\beta=2\sin(\frac{\pi}{2}\nu)>0$.

Because $s$ belongs to principal Riemann surface, the set $\Omega$ implies that
$\text{Re}(s^{\nu})\geq0$. Therefore, $\theta\in\mathbf{\Theta}(\Delta,\Sigma)$ implies that $\text{Re}(\theta)\geq0$.

Let $\zeta=\begin{bmatrix}
f\\g
\end{bmatrix}$. Note that i) satisfies if and only if either $f=\theta g\ (\theta\in\mathbb{C})$ or $g=0\ (\theta=\infty)$.

For statement ii), by some basic algebraic calculation, we get
\begin{align}
& \zeta^*\left[(\Delta+\Sigma)\otimes U\right]\zeta\notag\\
= & (\alpha+\beta)(g^*Uf+f^*Ug) \notag \\
= & (\alpha+\beta)\text{tr}\left[(fg^*+gf^*)U\right]\geq0
\label{prflemZetaReS1}
\end{align}

Note that $\alpha+\beta>0$. Inequality (\ref{prflemZetaReS1}) holds for all $U\in\mathcal{H}_n,U>0$, which implies that
\begin{equation}
fg^*+gf^*\geq0
\end{equation}

According to lemma \ref{lemFGgeq0}, statement i) is equivalent to statement ii) ($g\neq0$).

When $g=0$, it's obvious that statement i) is equivalent to statement ii). This ends the proof.
\end{proof}

Now, we are ready to state FGKYP for $H_{\infty}$ norm.

\begin{theorem}\label{thm01Spro}
Let matrices $\Pi\in\mathcal{H}_{n+m}$, $N\in\mathbb{C}^{2n\times(n+m)}$ be given.
$\Xi_{\theta}$ is defined in (\ref{defXis}) and $S_{\theta}$ is defined as
$S_{\theta}\triangleq(\Xi_{\theta}N)_{\perp}$. If the region $\mathbf{\Theta}(\Delta,\Sigma)$
defined by (\ref{defConvexD}) represents the region $\Omega=\{s^{\nu}\mid s\in\mathbb{C},
\text{Re}(s)\geq0,0<\nu\leq1\}$ on the complex plane and $s$ belongs to principal Riemann surface,
then the following statements are equivalent

i) $S_{\theta}^{\ast}\Pi S_{\theta}<0,$ $\forall \theta\in \mathbf{\Theta}$;

ii) There exists $U\in \mathcal{H}_{n}$ such that $U>0$ and%
\begin{equation}
N^{\ast}\left[(\Delta+\Sigma) \otimes U\right]N+\Pi<0
\end{equation}
\end{theorem}

\begin{proof}
Define set $\mathcal{F}$
\begin{equation}
 \mathcal{F}\triangleq \{N^{\ast}\left[(\Delta+\Sigma) \otimes U\right]N:U\in\mathcal{H}_n,U>0\}
 \label{prfThm01equ1}
\end{equation}

Let $\mathcal{G}_1$ defined by (\ref{defG1s}), and $\mathcal{G}_2$ be defined to be $\mathcal{G}_1$ in (\ref{DefFG}) with
(\ref{prfThm01equ1}). Then, for a nonzero vector $\eta$
\begin{align*}
& \eta \eta^{\ast}\in \mathcal{G}_{1}\\
\Leftrightarrow & \Xi_{\theta}N\eta=0\text{ for
some }\theta\in \mathbf{\Theta}(\Delta,\Sigma)\\
\Leftrightarrow & \eta^* N^*\left[(\Delta+\Sigma)\otimes U\right]N\eta\geq0\\
           & \text{for all }U\in\mathcal{H}_n,U>0\\
\Leftrightarrow & \eta \eta^{\ast}\in \mathcal{G}_{2}
\end{align*}
where the first and third equivalences easily follow from the definitions, and the second equivalence
holds due to Lemma \ref{lemZetaReS}.

The last step is to prove the set $\mathcal{F}$ is rank-one separable. Note that, if the region $\mathbf{\Theta}(\Delta,\Sigma)$
represents the region $\Omega$, there holds $\alpha+\beta>0$. We get
\[
(\Delta+\Sigma) \otimes U=\begin{bmatrix}
0 & (\alpha+\beta)U\\
(\alpha+\beta)U & 0
\end{bmatrix}
\]

Thus $(\Delta+\Sigma) \otimes U$ is rank-one separable according to set $\mathcal{F}_X$ in (\ref{DefFx}).
Therefore, according to Lemma \ref{lemRankone}, the set $\mathcal{F}$ is rank-one separable. Finally, i) is
equivalent to ii) due to Lemma \ref{LemSProc}. This ends the proof.
\end{proof}

Now, we can check the $H_{\infty}$ norm by LMI.

\begin{theorem}\label{thm01Hinf}
Consider the FOS with fractional order $0<\nu\leq1$ and its transfer function $G(s)$ in (\ref{equTransFOS}).
Given a prescribed $H_{\infty}$ performance bound $\delta>0$, then
$\left \Vert G(s)\right \Vert_{H_{\infty}}<\delta$ holds if and only if there exists $U\in\mathcal{H}_n,U>0$
such that the following LMI holds
\begin{equation}
\begin{bmatrix}
Sym(A^TU)\sin{(\frac{\pi}{2}\nu)} & UB\sin{(\frac{\pi}{2}\nu)} & C^T\\
B^TU\sin{(\frac{\pi}{2}\nu)} & -\delta I & D^T\\
C & D & -\delta I
\end{bmatrix}<0 \label{inequThm01H}
\end{equation}
\end{theorem}

\begin{proof}
Let the region $\mathbf{\Theta}(\Delta,\Sigma)$
defined by (\ref{defConvexD}) represent the region $\Omega=\{s^{\nu}\mid s\in\mathbb{C},
\text{Re}(s)\geq0,0<\nu\leq1\}$ on the complex plane, where $\Delta=\begin{bmatrix}
0 & \alpha\\
\overline{\alpha} & 0
\end{bmatrix}$ and $\Sigma=\begin{bmatrix}
0 & \beta\\
\overline{\beta} & 0
\end{bmatrix}$. Then there holds $\alpha+\beta=2\sin{(\frac{\pi}{2}\nu)}$.

Let $K_{\theta}=(\theta I-A)^{-1}B$. By some basic matrix calculations, we have
\begin{align*}
& \left\Vert G(s)\right\Vert_{H_{\infty}}<\delta\\
\Leftrightarrow & G^*(s)G(s)-\delta^2 I< 0, \text{Re}(s)\geq0\\
\Leftrightarrow & \begin{bmatrix}
K_{\theta}\\I
\end{bmatrix}^*\Pi\begin{bmatrix}
K_{\theta}\\I
\end{bmatrix},\forall\theta\in\mathbf{\Theta}
\end{align*}
where $\Pi=\begin{bmatrix}
C^TC & C^TD\\
D^TC & D^TD-\delta^2 I
\end{bmatrix}$

According to Theorem \ref{thm01Spro}, there exists a matrix $U\in\mathcal{H}_n,U>0$, such that
\begin{align*}
\begin{bmatrix}
A & B \\
I & 0
\end{bmatrix}^T\begin{bmatrix}
0 & (\alpha+\beta)U\\
(\alpha+\beta)U & 0
\end{bmatrix}\begin{bmatrix}
A & B \\
I & 0
\end{bmatrix}+\Pi<0
\end{align*}

The above LMI can be further simplified as
\begin{align*}
\begin{bmatrix}
2Sym(A^TU)\sin{(\frac{\pi}{2}\nu)} & 2UB\sin{(\frac{\pi}{2}\nu)} \\
2B^TU\sin{(\frac{\pi}{2}\nu)} & -\delta^2 I
\end{bmatrix}+\begin{bmatrix}
C & D
\end{bmatrix}^T\begin{bmatrix}
C & D
\end{bmatrix}<0
\end{align*}

Rescaling $U$ and utilizing the Schur complement theorem, we finally get (\ref{inequThm01H}). This ends the proof.
\end{proof}

\begin{remark}
When $\nu=1$, the condition is as same as the condition of $H_\infty$ for integer order system\cite{Boyd1994Linear}.
Meanwhile, Liang\cite{Liang2015Bounded} gives a sufficient condition of $H_\infty$ with fractional order $0<\nu<1$. And if the unknown matrix $U$ in Theorem \ref{thm01Hinf} is an arbitrary matrix, the theorems between Liang's and ours are equivalent.
\end{remark}


\subsection{For fractional order $1<\nu<2$}
The following gives theorems of $H_{\infty}$ norm for FOS with fractional order $1<\nu<2$.

\begin{lemma}\label{lemFG_geq0}
\cite{Rantzer1996Kalman}
Let $F,G$ be complex matrices of the same size. Then
\[
FG^{\ast}+GF^{\ast}\geq0
\]
if and only if there exists a matrix $U$ such that $UU^{\ast}\leq I$ and $F(I+U)=G(I-U)$.
\end{lemma}

Thus, we get the following lemma.
\begin{lemma}\label{lemfgIms}
Let $f,g\in\mathbb{C}^n$ and $g\neq0$. Then
\begin{equation}
fg^*+gf^*\geq0\Leftrightarrow f=\text{j}\theta g\ \text{for some }\theta\in\mathbb{C}\text{ with Im}(\theta)\leq0
\end{equation}
\end{lemma}

\begin{proof}
Application of Lemma \ref{lemFG_geq0} with $(1-U)/(1+U)=\text{j}\theta$ gives the desired result.
\end{proof}

\begin{lemma}\label{lemZetaImS}
Let $\mathbf{\Theta}(\Delta,\Sigma)$ be defined by (\ref{defConvexD}), $\Xi_{\theta}$ by (\ref{defXis}) and
$\zeta$ is a given vector. If the region $\mathbf{\Theta}(\Delta,\Sigma)$
defined by (\ref{defConvexD}) represents the region $\Omega$ on the complex plane. $\Omega$
and $\overline{\Omega}$ are symmetrical with respect to the real axis on the complex plane and satisfy
$\Omega\cup\overline{\Omega}=\{s^{\nu}\mid s\in\mathbb{C}, \text{Re}(s)\geq0,1<\nu<2\}$ and
$\Omega\cap\overline{\Omega}=\{s\mid s\in\mathbb{C}, \text{Im}(s)=0 \}$ and $s$ belongs to principal Riemann surface,
then the following statements are equivalent.

i) $\Xi_{\theta}\zeta=0$ for some $\theta\in\mathbf{\Theta}$ with $\text{Im}(\theta)\leq0$;

ii) $\zeta^*\{\left[T^*_0(\Delta+\Sigma^T)T_0\right]\otimes U\}\zeta\geq0$ for all $U\in\mathcal{H}_n,U>0$, where $T_0=\begin{bmatrix}
\text{e}^{\text{j}\frac{\pi}{4}} & 0 \\
0 & \text{e}^{-\text{j}\frac{\pi}{4}}
\end{bmatrix}$
\end{lemma}

\begin{proof}
Let $\Delta=\begin{bmatrix}
0 & a+\text{j}c\\
a-\text{j}c & 0
\end{bmatrix}$ and $\Sigma=\begin{bmatrix}
0 & b+\text{j}d\\
b-\text{j}d & 0
\end{bmatrix}$, ($\alpha=a+\text{j}c$, $\beta=b+\text{j}d$), and $a,b,c,d\in\mathbb{R}$.
Because $s$ belongs to principal Riemann surface, the set $\Omega$ can be chosen so that
$\text{Im}(s^{\nu})\leq0$. Therefore, $\theta\in\mathbf{\Theta}(\Delta,\Sigma)$ implies that $\text{Im}(\theta)\leq0$.
If $\mathbf{\Theta}$ represent the region $\Omega$ and
$\theta=x+\text{j}y\in\mathbf{\Theta},x,y\in\mathbb{R}$, then
\[
\mathbf{\Theta}=\{\theta=x+\text{j}y\mid \sin(\frac{\pi}{2}\nu)x+\cos(\frac{\pi}{2}\nu)y\geq0,\ \cos(\frac{\pi}{2}\nu)y\geq0\}
\]
i.e. $a=\sin(\frac{\pi}{2}\nu)>0,b=0,c=d=\cos(\frac{\pi}{2}\nu)$.

Therefore, when $\mathbf{\Theta}$ represent the region $\Omega$, there holds $\alpha+\overline{\beta}=2\sin(\frac{\pi}{2}\nu)>0$.

Let $\zeta=\begin{bmatrix}
f\\
\text{j}g
\end{bmatrix}$. Note that i) satisfies if and only if either $f=\text{j}\theta g\ (\theta\in\mathbb{C})$ or $g=0\ (\theta=\infty)$.

For statement ii), by some basic algebraic calculation, we get
\begin{align}
& \zeta^*\{\left[T^*_0(\Delta+\Sigma^T)T_0\right]\otimes U\}\zeta\notag\\
= & (\alpha+\overline{\beta})(g^*Uf+f^*Ug) \notag \\
= & (\alpha+\overline{\beta})\text{tr}\left[(fg^*+gf^*)U\right]\geq0
\label{prflemZetaImS1}
\end{align}

Note that $\alpha+\overline{\beta}>0$. Inequality (\ref{prflemZetaImS1}) holds for all $U>0$, which implies that
\begin{equation}
fg^*+gf^*\geq0
\end{equation}

According to lemma \ref{lemfgIms}, statement i) is equivalent to statement ii) ($g\neq0$).

When $g=0$, it's obvious that statement i) is equivalent to statement ii). This ends the proof.
\end{proof}

\begin{theorem}\label{thm12Spro}
Let matrix $\Pi\in\mathcal{H}_{n+m}$ be given. The region $\mathbf{\Theta}(\Delta,\Sigma)$
defined by (\ref{defConvexD}) represents the region $\Omega$. Region $\Omega$ and $\overline{\Omega}$
are symmetrical with respect to the real axis on the complex plane and satisfy
 $\Omega\cup\overline{\Omega}=\{ s^{\nu}\mid s\in\mathbb{C},\text{Re}(s)\geq0,1<\nu<2 \}$ and
$\Omega\cap\overline{\Omega}=\{s\mid s\in\mathbb{C}, \text{Im}(s)=0 \}$.
Meanwhile, $s$ belongs to the principal Riemann surface.
$\Xi_{\theta}$ is defined by (\ref{defXis}), $N$ by (\ref{defNs}) and $S_{\theta}$ is defined as
$S_{\theta}\triangleq(\Xi_{\theta}N)_{\perp}$. Then, the following statements are equivalence

i) $S_{\theta}^{\ast}\Pi S_{\theta}<0,$ $\forall \theta\in \mathbf{\Theta}(\Delta,\Sigma)$;

ii) There exists $U\in \mathcal{H}_{n}$ such that $U>0$ and
\begin{equation}
N^{\ast}\{\left[T^*_0(\Delta+\Sigma^T)T_0\right]\otimes U\}N+\Pi<0
\end{equation}
where $T_0=\begin{bmatrix}
\text{e}^{\text{j}\frac{\pi}{4}} & 0 \\
0 & \text{e}^{-\text{j}\frac{\pi}{4}}
\end{bmatrix}$.
\end{theorem}

\begin{proof}
Define set $\mathcal{F}$
\begin{equation}
 \mathcal{F}\triangleq \{N^{\ast}((T^*_0(\Delta+\Sigma^T)T_0)\otimes U)N:U\in\mathcal{H}_n,U>0\}
 \label{prfThm12equ1}
\end{equation}

Let $\mathcal{G}_1$ defined by (\ref{defG1s}), and $\mathcal{G}_2$ be defined to be $\mathcal{G}_1$ in (\ref{DefFG}) with
(\ref{prfThm12equ1}). Then, for a nonzero vector $\eta$
\begin{align*}
& \eta \eta^{\ast}\in \mathcal{G}_{1}\\
\Leftrightarrow & \Xi_{\theta}N\eta=0\text{ for
some }\theta\in \mathbf{\Theta}(\Delta,\Sigma)\\
\Leftrightarrow & \eta^* N^*((T^*_0(\Delta+\Sigma^T)T_0)\otimes U)N\eta\geq0\\
           & \text{for all }U\in\mathcal{H}_n,U>0\\
\Leftrightarrow & \eta \eta^{\ast}\in \mathcal{G}_{2}
\end{align*}
where the first and third equivalences easily follow from the definitions, and the second equivalence
holds due to Lemma \ref{lemZetaImS}.

The last step is to prove the set $\mathcal{F}$ is rank-one separable.  Note that $(T^*_0(\Delta+\Sigma^T)T_0)\otimes U=(T_0\otimes I)^*((\Delta+\Sigma^T)\otimes U)
(T_0\otimes I)$. Because the region $\mathbf{\Theta}(\Delta,\Sigma)$
represents the region $\Omega$, there holds $\overline{\alpha}+\beta=\alpha+\overline{\beta}>0$. We get
\[
(\Delta+\Sigma^T) \otimes U=\begin{bmatrix}
0 & (\alpha+\overline{\beta})U\\
(\overline{\alpha}+\beta)U & 0
\end{bmatrix}
\]

Thus $(\Delta+\Sigma^T) \otimes U$ is rank-one separable according to set $\mathcal{F}_X$ in (\ref{DefFx}).
Therefore, according to Lemma \ref{lemRankone}, the set $\mathcal{F}$ is rank-one separable. Finally, i) is
equivalent to ii) due to Lemma \ref{LemSProc}. This ends the proof.
\end{proof}

The following check the $H_\infty$ of FOS with fractional order $1<\nu<2$.

\begin{theorem}\label{thm12Hinf}
Consider the FOS with fractional order $1<\nu<2$ and its transfer function $G(s)$ in (\ref{equTransFOS}).
Given a prescribed $H_{\infty}$ performance bound $\delta>0$, then
$\left \Vert G(s)\right \Vert_{H_{\infty}}<\delta$ holds if and only if there exists a matrix $U\in\mathcal{H}_n,U>0$
such that the following LMI holds
\begin{equation}
\begin{bmatrix}
Sym(\text{j}UA)\sin(\frac{\pi}{2}\nu) & \text{j}UB\sin(\frac{\pi}{2}\nu) & C^T\\
-\text{j}B^TU\sin(\frac{\pi}{2}\nu) & -\delta I & D^T\\
C & D & -\delta I
\end{bmatrix}<0 \label{inequThm12H}
\end{equation}
\end{theorem}

\begin{proof}
Let region $\Omega$ and $\overline{\Omega}$
are symmetrical with respect to the real axis on the complex plane and satisfy
 $\Omega\cup\overline{\Omega}=\{ s^{\nu}\mid s\in\mathbb{C},\text{Re}(s)\geq0,1<\nu<2 \}$
 and $\Omega\cap\overline{\Omega}=\{s\mid s\in\mathbb{C},\text{Im}(s)=0 \}$.
 It's a fact that there must hold
\[
 \left\Vert G(s) \right\Vert_{H_{\infty}}=\underset{\text{Re}(s)\geq0}{\sup}\sigma_{\max}(G(s))=
 \underset{s\in\Omega}{\sup}\sigma_{\max}(G(s)).
\]
 This just follows from the maximum modulus principle and the complex conjugate symmetry of $G(s)$.

 The region $\mathbf{\Theta}(\Delta,\Sigma)$
defined by (\ref{defConvexD}) represents the region $\Omega$, where $\Delta$ and $\Sigma$
\[
\Delta=\begin{bmatrix}
0 & \alpha\\
\overline{\alpha} & 0
\end{bmatrix},\ \Sigma=\begin{bmatrix}
0 & \beta\\
\overline{\beta} & 0
\end{bmatrix}
\]
and $\alpha+\overline{\beta}=2\sin(\frac{\pi}{2}\nu)$.

Similarly to the proof of Theorem \ref{thm01Hinf} with $\alpha+\overline{\beta}=2\sin(\frac{\pi}{2}\nu)$ and due to the
Theorem \ref{thm12Spro}, we can get the result.
\end{proof}

\begin{remark}
Liang\cite{Liang2015Bounded} also proves a sufficient and necessary condition of
$H_\infty$ with fractional order $1<\nu<2$. The two conditions between Liang's and
ours are equivalent because $jU\sin(\frac{\pi}{2}\nu)$ can be regarded as an arbitrary complex matrix.
\end{remark}

\section{Numerical examples}

In order to use the LMI tools of Matlab, the following fact should be introduced.

\begin{fact}
A Hermitian matrix $H<0$ holds if and only if the following real LMI holds
\begin{equation*}
\begin{bmatrix}
\mathrm{Re}(H) & \mathrm{Im}(H)\\
\mathrm{Im}(H)^T & \mathrm{Re}(H)
\end{bmatrix}<0
\end{equation*}
\end{fact}

\begin{example}
The following shows the $L_\infty$ of FOS with low frequency range.

Consider the transfer function $G(s)$ in (\ref{equTransFOS}) with the parameters
described as following.

\begin{align*}
A&=
\begin{bmatrix}
-12.1 & 2.3\\
2.37 & -16.2
\end{bmatrix},
B=\begin{bmatrix}
-2\\
1.2
\end{bmatrix},\\
C&=\begin{bmatrix}
1.5 & 1.9
\end{bmatrix},
D=0.8,\\
\nu & =0.6,\ \delta=0.9,\ 0\leq\omega\leq 100
\end{align*}

The maximum singular values are shown in figure \ref{FigSin}. It shows that the $L_\infty$ norm is less than 0.77 in the frequency range. Due to Theorem \ref{LMIforLowF}, solving the LMI (\ref{EqLFFOS}) via Matlab, we get
\begin{align*}
U=\begin{bmatrix}
4.4908 & 7.3472\\
7.3472 & 13.5229
\end{bmatrix},
V=\begin{bmatrix}
0.0772 & 0.1525\\
0.1525 & 0.4045
\end{bmatrix}
\end{align*}

This implies that $L_\infty<0.9$ is convinced. According to figure \ref{FigSin}, $L_\infty<0.77<0.9$, which means Theorem \ref{LMIforLowF} is correct.
However, when we set $\delta=0.6$, LMI (\ref{EqLFFOS}) cannot be solved because
0.6 is less than the max value shown in figure \ref{FigSin}. It verifies that
Theorem \ref{LMIforLowF} is correct.

\end{example}

\begin{example}
The following gives an example of the Theorem \ref{thm01Hinf}.

Consider the transfer function $G(s)$ in (\ref{equTransFOS}) with the parameters
described as following.

\begin{align*}
A &=\begin{bmatrix}
-1.9 & 1.3 \\
0.6 & -1.5
\end{bmatrix},
B=\begin{bmatrix}
-1.8\\
2.7
\end{bmatrix},\\
C &=\begin{bmatrix}
2.2 & 3.1
\end{bmatrix},
D=0.2,\\
\nu &=0.7, \delta=9.2
\end{align*}

The eigenvalues of this system are shown in figure \ref{FigEig}. It shows
that the  system is stable.

Solving the LMI (\ref{inequThm01H}) via Matlab, we get

\begin{align*}
U=\begin{bmatrix}
1.6211 & 0.8928\\
0.8928 & 2.2315
\end{bmatrix}
\end{align*}

This implies that $H_\infty<9.2$ is convinced. However, when we set
$\delta=1.6$, the LMI (\ref{inequThm01H}) cannot be solved, which means
$H_\infty<1.6$ is not verified.

\end{example}

\begin{figure}
\centering
  \includegraphics[width=0.50\textwidth]{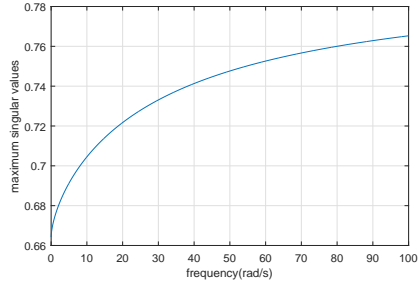}\\
  \caption{Max singular values corresponding to frequency}\label{FigSin}
\end{figure}

\begin{figure}
\centering
  \includegraphics[width=0.50\textwidth]{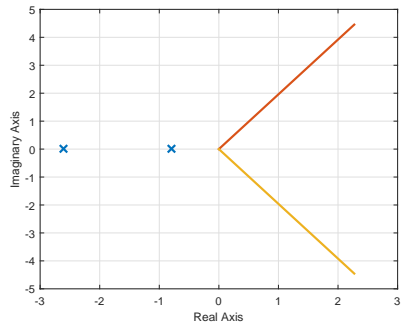}\\
  \caption{Eigenvalues of the FOS}\label{FigEig}
\end{figure}

\section{Conclusion}
In this paper, FGKYP is proved, which develops the GKYP into the fractional order system. $S$-procedure is used to bridge between the matrix inequality and frequency range. We prove the FGKYP for $L_\infty$ and $H_\infty$ of FOS, respectively. Based on the FGKYP, $L_\infty$ of FOS with different frequency range is proved. $H_\infty$ of FOS is proved and the FGKYP is different between fractional order $0<\nu\leq 1$ and
$1<\nu<2$. Examples are given to verify the theorems.

\bibliographystyle{unsrt}
\bibliography{H-infiniteControl}

\end{document}